\documentclass[12pt]{amsart}

\usepackage{amssymb}
\usepackage[dvipsnames]{xcolor}
\usepackage{a4wide}
\usepackage[colorinlistoftodos,prependcaption,textsize=tiny]{todonotes}
\usepackage[all,cmtip]{xy}

\usepackage{enumerate}

\usepackage{graphicx}
\usepackage{setspace}
\usepackage{tikz}
\usepackage{tkz-graph}
\usetikzlibrary{arrows}

\usepackage[bookmarks=true,
            bookmarksnumbered=true, breaklinks=true,
            pdfstartview=FitH, hyperfigures=false,
            plainpages=false, naturalnames=true,
            colorlinks=true,pagebackref=true,
            pdfpagelabels]{hyperref}

  \makeatletter
  \CheckCommand*\refstepcounter[1]{\stepcounter{#1}%
      \protected@edef\@currentlabel
       {\csname p@#1\endcsname\csname the#1\endcsname}%
  }
  \renewcommand*\refstepcounter[1]{\stepcounter{#1}%
    \protected@edef\@currentlabel
      {\csname p@#1\expandafter\endcsname\csname the#1\endcsname}%
  }
  \def\labelformat#1{\expandafter\def\csname p@#1\endcsname##1}
  \DeclareRobustCommand\Ref[1]{\protected@edef\@tempa{\ref{#1}}%
     \expandafter\MakeUppercase\@tempa
  }
  \makeatother

  \makeatletter
  \newcommand{\numberlike}[2]{%
     \expandafter\def\csname c@#1\endcsname{%
         \expandafter\csname c@#2\endcsname}%
  }
  \makeatother

  \def\DefaultNumberTheoremWithin{section}

  \theoremstyle{plain}
  \newtheorem{Lemma}{Lemma}
     \numberwithin{Lemma}{\DefaultNumberTheoremWithin}
     \labelformat{Lemma}{Lemma~#1}
  
     \numberwithin{Claim}{\DefaultNumberTheoremWithin}
     \numberlike{Claim}{Lemma}
     \labelformat{Claim}{Claim~#1}
  \newtheorem{Theorem}{Theorem}
     \numberwithin{Theorem}{\DefaultNumberTheoremWithin}
     \numberlike{Theorem}{Lemma}
     \labelformat{Theorem}{Theorem~#1}
  \newtheorem{Corollary}{Corollary}
     \numberwithin{Corollary}{\DefaultNumberTheoremWithin}
     \numberlike{Corollary}{Lemma}
     \labelformat{Corollary}{Corollary~#1}
  
     \numberwithin{Proposition}{\DefaultNumberTheoremWithin}
     \numberlike{Proposition}{Lemma}
     \labelformat{Proposition}{Proposition~#1}
  
     \numberwithin{Conjecture}{\DefaultNumberTheoremWithin}
     \numberlike{Conjecture}{Lemma}
     \labelformat{Conjecture}{Conjecture~#1}

  \theoremstyle{definition}
  \newtheorem{Definition}{Definition}
     \numberwithin{Definition}{\DefaultNumberTheoremWithin}
     \numberlike{Definition}{Lemma}
     \labelformat{Definition}{Definition~#1}

  \theoremstyle{definition}
  
     \numberwithin{Question}{\DefaultNumberTheoremWithin}
     \numberlike{Question}{Lemma}
     \labelformat{Question}{Question~#1}

  \theoremstyle{definition}
  
     \numberwithin{Problem}{\DefaultNumberTheoremWithin}
     \numberlike{Problem}{Lemma}
     \labelformat{Problem}{Problem~#1}

  \theoremstyle{remark}
  \newtheorem{Remark}{Remark}
     \numberwithin{Remark}{\DefaultNumberTheoremWithin}
     \numberlike{Remark}{Lemma}
     \labelformat{Remark}{Remark~#1}
  \theoremstyle{remark}

     \numberwithin{Example}{\DefaultNumberTheoremWithin}
     \numberlike{Example}{Lemma}
     \labelformat{Example}{Example~#1}
  
     \labelformat{Case}{Case~#1}
     \numberwithin{Case}{Lemma}
  
     \labelformat{Step}{Step~#1}
     \numberwithin{Step}{Lemma}

  \labelformat{equation}{(#1)}
  \labelformat{figure}{Figure~#1}
  \labelformat{section}{Section~#1}
  \labelformat{subsection}{Section~#1}
  \labelformat{subsubsection}{Section~#1}

  \def\eqref{\ref}

  \def\Ker{\mathrm{Ker}}
  \def\Im{\mathrm{Im}}

  \def\Hom{\mathcal{H}}

  \def\CC{{\mathcal{C}}}

  \def\LL{{\mathcal{L}}}
   \def\DD{{\mathcal{D}}}
   \def\push{{\mathrm{push}}}

   \def\Con{{\normalfont {\textrm{Con}}}}

\begin{document}

\title[Homology of Cubical Sets with Connections]        % show up at every other page
{Homology Groups of Cubical Sets with Connections}

\author[H. Barcelo]{H\'{e}l\`{e}ne Barcelo}

\address[H\'{e}l\`{e}ne Barcelo]
{
The Mathematical Sciences Research Institute, 17 Gauss Way, Berkeley, CA 94720, USA
}
\email{hbarcelo@msri.org}

\author[C. Greene]{Curtis Greene}

\address[Curtis Greene]
{
Haverford College, Haverford, PA 19041, USA
}
\email{cgreene@haverford.edu}

\author[A.S. Jarrah]{Abdul Salam Jarrah}

\address[Abdul Jarrah]
{
Department of Mathematics and Statisticss, American University of Sharjah, PO Box 26666, Sharjah, United Arab Emirates
}
\email{ajarrah@aus.edu}

\author[V.Welker]{Volkmar Welker}

\address[Volkmar Welker]
{
Fachbereich Mathematik und Informatik, Philipps-Universit\"at, 35032 Marburg, Germany
}
\email{welker@mathematik.uni-marburg.de}

\thanks{This material is based upon work supported by the National Science Foundation under Grant No. DMS-1440140 while the authors were in residence at the
Mathematical Sciences Research Institute in Berkeley, California, USA}

%\date{\today}

\begin{abstract}
Toward defining commutative cubes in all dimensions, Brown and Spencer introduced the notion of ``connection'' as a new kind of degeneracy. In this paper, for a cubical set with connections, we show that the connections generate an acyclic subcomplex of the chain complex of the cubical set. In particular, our results show that the homology groups of a cubical set with connections are independent of whether we normalize by the connections or we do not, that is, connections do not contribute to any nontrivial cycle in the homology groups of the cubical set.
\end{abstract}

%\subjclass{55N35, 55U15}
%\keywords{Cubical Homology, Connection, Cubical Set}

\maketitle

%\noindent
%{\em LATEX DATE: \today\quad FILE: \jobname}
%\bigskip

%%%%%%%%%%%%%%%%%%%%%%%%%%%%%%%%%%%%%%%%%%%%%%%%%%%%%%%%%%%%%%%
\section{Introduction}
%%%%%%%%%%%%%%%%%%%%%%%%%%%%%%%%%%%%%%%%%%%%%%%%%%%%%%%%%%%%%%%

Cubical sets stemmed naturally from the development of homology theory of various spaces.
Instead of simplices, cubes were, for the first time, used by Serre to develop (co)homology theory for fiber spaces  \cite{Serre:51},  and Eilenberg and MacLane \cite{EM} developed the singular, cubical homology theory of topological spaces.
Massey's classical book \cite{Massey80} presents a comprehensive treatment of singular homology using the cubical approach. 

Kan introduced and studied abstract cubical sets for the purpose of developing a general homotopy theory, see \cite{Kan1}. Cubical sets come with a singular homology theory \cite[Section 14.7]{Brown} and a geometric realization \cite[Definition 11.1.11]{Brown}.
Federer \cite[Theorem 3.9.12]{Federer} showed that the singular homology groups of a cubical set and that of its geometric realization are isomorphic. 

Toward the development of a general abstract homotopy theory, Brown and Spencer \cite{BroSpen} identified the need, in higher dimensions, for what they call ``commutative'' cubes, and introduced a new kind of degeneracy which they call ``connections.''   Cubical sets with connections were then introduced and studied by Brown and Higgins in \cite{BrownHiggins2}. The recent paper \cite{Brown18} explains the origin of the notion of connection as well as the need for it.

Not all cubical sets admit connections. However, cubical sets with connections have been shown to have many desirable properties \cite{BrownHiggins}, and have characteristics similar to that of simplicial sets \cite{GranMau}. For examples, cubical abelian groups with connections are equivalent to chain complexes \cite{brown2003}, and cubical groups with connections are Kan fibrant  \cite{Tonks:92}, a property shared with simplicial sets. Recently, in  \cite{Maltsinioti:09}, it was shown that cubical sets with connections form a strict test category. In particular, the geometric realization of the product of cubical sets with connections has the ``right'' homotopy type; a property that cubical set (without connections) do not have in general.

In this note we study the singular homology groups of cubical sets with connections. We were originally motivated by computational considerations encountered in \cite{BGJW}. Since the chain groups are very large, we explored cutting down the size of the
the chain complex by dropping connection cubes. For this purpose, we investigate the contribution of connections to the nontrivial cycles in the homology groups. We do so by studying the relations between the singular cubical differential, the face maps, the degeneracy maps and the connections maps. 
This study culminates in \ref{thm:bd} from which we then deduce
in \ref{cor:subcomplex} that connections generate a chain subcomplex of the singular chain complex of the cubical set. Furthermore, using a chain homotopy given in 
\ref{thm:homotopy} we deduce in \ref{cor:final} 
that the homology groups of this subcomplex are trivial. 
In particular, the quotient
of the singular chain complex of the cubical set by the
subcomplex generated by the connection cubes computes
the same homology as the singular chain complex itself.

In an appendix we provide the arguments showing that 
this quotient complex indeed is the cellular chain
complex of the canonical CW-structure on the geometric
realization of a cubical set with connections (see
\ref{thm:final}). In particular, for a cubical set with connections, we state 
in \ref{cor:superfinal} that the singular
homology groups of the geometric realizations with and without connection identifications coincide. 

The latter result is also a consequence of a result by
Antolini \cite{Antolini:02}, who states that the
two realizations are homotopy equivalent. 
Since we consider Antolini's arguments hard to penetrate, we
see some value of our down to earth derivation.
%%%%%%%%%%%%%%%%%%%%%%%%%%%%%%%%%%%%%%%%%%%%%%%%%%%%%%%%%%%%%%%
\section{Background and Notations}
%\section{Definitions and Notations}
%%%%%%%%%%%%%%%%%%%%%%%%%%%%%%%%%%%%%%%%%%%%%%%%%%%%%%%%%%%%%%%

In this section we recall the definition of a cubical set with connections and the homology theory of cubical sets. Then we give two examples of such sets to demonstrate the motivation for this study. 

Throughout the paper,  $R$ denotes a commutative ring with unit which shall be the ring of coefficients.
For any positive integer $n$, let $[n] := \{1,\dots,n\}$.

%%%%%%%%%%%%%%%%%%%%%%%%%%%%%%%%%%%%%%%%%%%%%%%%%%%%%%%%%%%%%%%
%\subsection*{Cubical Sets}

\begin{Definition}[\cite{Kan1}] \label{def:cubicalset}
A \textit{cubical set}  $K$ is a collection of sets $\{K_n\}_{n\geq 0}$ together with, for each $n \geq 1$ and each $i\in [n]$,
\begin{enumerate} %[label=(\roman*)]
    \item two maps $f_i^+, f_i^- : K_n \longrightarrow K_{n-1}$, which are called \emph{face maps}, and
   \item  a  map $\varepsilon_i:K_{n-1} \longrightarrow K_n$, which is called a \emph{degeneracy map},
 \end{enumerate}
   satisfying the following relations: For $\alpha,\beta \in\{+,-\}$,
   %\begin{align}
   \begin{enumerate}[(i)]
    { \setstretch{1.3}
    \item  $f_i^\alpha f_j^\beta =  f_{j-1}^\beta f_i^\alpha  \hspace{0.8cm} \mbox{ if } i<j.$
     \item $\varepsilon_i \varepsilon_j = \varepsilon_{j+1} \varepsilon_i \hspace{1cm} \mbox{ if } i \leq j. $
    \item  $f_i^\alpha \varepsilon_j = \left\{ \begin{array}{ll}
                                        \varepsilon_{j-1} f_i^\alpha & \mbox{if  } i < j ;\\
                                        \varepsilon_{j} f_{i-1}^\alpha & \mbox{if } i > j; \\
                                       id & \mbox{if } i = j.\end{array} \right.$ \label{def:cc3}
    }
    \end{enumerate}
\end{Definition}
In a cubical set $K$, an  element $\sigma \in K_n$ is called a \textit{singular $n$-cube}. 
A singular $n$-cube $\sigma$ is said to be degenerate if  $\sigma = \varepsilon_if_i^+\sigma$ for some $i\in [n]$.
Otherwise, $\sigma$ is called \textit{non-degenerate}.

\begin{Definition}[\cite{AlAgl}] \label{def:cub-conn}
A \textit{cubical set with connections} is a cubical set $K$ together with, for $n\geq 1$
 and each $i \in [n]$, two additional maps (called \emph{connections})
\[
\Gamma_i^+, \Gamma_i^- :K_n \longrightarrow K_{n+1}.
\]
such that, for $\alpha,\beta \in \{+,-\}$ and $i,j\in[n]$, the following relations are satisfied:
\begin{enumerate}[(i)]
 { \setstretch{1.3}
\item 
     $ \Gamma_i^\alpha \Gamma_j^\beta =\Gamma_{j+1}^\beta \Gamma_i^\alpha \hspace{1cm} \mbox{ if } i \leq j$.
\item
     $ \Gamma_i^\alpha \varepsilon_j = \left\{ \begin{array}{ll}
                                        \varepsilon_{j+1} \Gamma_i^\alpha & \mbox{if  } i < j ;\\
                                        \varepsilon_{j} \Gamma_{i-1}^\alpha & \mbox{if } i > j; \\
                                       \varepsilon_i^2= \varepsilon_{i+1}\varepsilon_i & \mbox{if } i = j.
                                       \end{array} \right.	$											
\item \label{def:ccc3}
       $ f_i^\alpha \Gamma_j^\beta = \left\{ \begin{array}{ll}
                                \Gamma_{j-1}^\beta f_i^\alpha & \mbox{if  } i < j;\\
                                \Gamma_{j}^\beta f_{i-1}^\alpha & \mbox{if } i > j+1;\\
                                id	&  \mbox{if } i=j, j+1, \alpha=\beta;	\\
                                 \varepsilon_if_i^\alpha & \mbox{ if } i=j, j+1, \alpha \neq \beta.
            \end{array} \right.$
 }
\end{enumerate}
\end{Definition}

%%%%%%%%%%%%%%%%%%%%%%%%%%%%%%%%%%%%%%%%%%%%%%%%%%%%%%%%%%%%%%%
\subsection*{Homology Groups of Cubical Sets}

Let $K$ be a cubical set and let $R$ be the ring of coefficients. For each $n \geq 0$, let
$\LL_n(K)$ be the free $R$-module generated by the singular $n$-cubes with coefficients from $R$, that is,
\[
\LL_n(K) := \{ \sum_{\sigma \in S} r_\sigma \sigma :S \mbox{ finite subset of } K_n \mbox{ and } r_\sigma \in R \}.
\]
For $n >0$, define the map $\partial_n : \LL_n(K) \longrightarrow \LL_{n-1}(K)$  such that, for each singular $n$-cube $\sigma$,
\[
\partial_n(\sigma) = \sum_{i=1}^n (-1)^i (f_i^-\sigma - f_i^+\sigma)
\]
and extend linearly to all elements of  $\LL_n(K)$. Furthermore, define the map $\partial_0:\LL_0(K) \longrightarrow \LL_{-1}(K) (=\{0\})$ to be the zero map, that is
$\partial_0(\sigma)=0$ for all $\sigma \in \LL_0$.

For each $n\geq 1$, let $\DD_n(K)$ be the $R$-submodule of
$\LL_n(K)$ that is generated by all degenerate singular $n$-cubes,
and let $\CC_n(K)$ be the free $R$-module
$\LL_n(K)/\DD_n(K)$, whose elements are called $n$-chains.
Clearly, the cosets of non-degenerate singular $n$-cubes freely generate $\CC_n(K)$.

Using \ref{def:cubicalset}(\ref{def:cc3}), it is easy to check that
$\partial_n [\DD_n(K)] \subseteq \DD_{n-1}(K)$  and, for $n\geq 1$,  $\partial_{n-1} \partial_n = 0$, see \cite{BCW, Massey80}.
Hence,
$\partial_n: \CC_n(K) \longrightarrow \CC_{n-1}(K)$ is a boundary operator,
and  $\CC(K) = (\CC_\bullet(K),\partial_\bullet)$ is a
chain complex of free $R$-modules. We call $\CC(K)$ the \textit{non-degenerate chain complex} of the cubical set $K$.

The homology groups of $K$ are defined to be the homology groups of the chain complex $\CC(K)$, that is,   $\Hom_n(K) := \Ker(\partial_n)/\Im(\partial_{n+1})$, see \cite{Kan1}.
 For more information about the homology and homotopy of cubical sets see \cite[Sections 14.7 and 13.1]{Brown}.

\subsection*{Cubical Sets of Topological Spaces}%\cite[Section 11.3]{Brown}.
Let $X$ be a topological space, and, for $n \geq 0$, let $I^n$ be the geometric $n$-dimensional cube, that is, $I^n := \{(x_1,\dots,x_n) : x_i \in [0,1], i \in[n]\}$ with the standard topology. Define $KX_n$ to be the set of all continuous maps $\sigma: I^n \longrightarrow X$.
For each $i \in [n]$ and $\sigma \in KX_n$, define face maps $f_i^+\sigma ,f_i^-\sigma \in KX_{n-1}$ such that, for $(a_1,\dots,a_{n-1}) \in I^{n-1}$,
\begin{align*}
(f_i^+\sigma)(a_1,\dots,a_{n-1}) & := \sigma(a_1,\dots,a_{i-1},1, a_{i},\dots,a_{n-1}), \\
(f_i^-\sigma)(a_1,\dots,a_{n-1}) &:= \sigma(a_1,\dots,a_{i-1},0, a_{i},\dots,a_{n-1}).
\end{align*}
Also, define $\varepsilon_i\sigma \in  KX_{n+1}$ such that, for $(a_1,\dots,a_{n+1}) \in I^{n+1}$,
\[
(\varepsilon_i\sigma)(a_1,\dots,a_{n+1}) := \sigma(a_1,\dots,a_{i-1},a_{i+1},\dots,a_{n+1}).
\]
It is easy to check that $KX := \{KX_n\}_{n\geq 0}$ along with the face maps $f_i^\pm$ and degeneracy maps $\varepsilon_i$ is a cubical set.

Furthermore, $KX$ is a cubical set with connections defined as follows. For each $i\in [n]$, set 

\[
\Gamma_i^{\varepsilon}\sigma(a_1,\dots, a_{n+1}) := \sigma(a_1,\dots,a_{i-1}, m_\varepsilon(a_i,a_{i+1}),a_{i+2},\dots, a_{n+1})
\]
where
\[m_\varepsilon(x,y)=
\left\{ \begin{array}{ll}
         \min(x,y) & \mbox{if }\varepsilon = +;\\
        \max(x,y) & \mbox{if } \varepsilon = -.\end{array} \right.
\]
 The set $KX$ was initially constructed by Eilenberg and Mac Lane \cite{EM} and was used to define the cubical singular homology groups of $X$, which turned out to be the same as the (classical) singular homology groups of $X$, that is, $H_n(X)= H_n(KX)$ for all $n$, see \cite[Section 2, Chapter II]{Massey80}. Furthermore, the geometric realization $|KX|$ of $KX$ and $X$ are weakly homotopy equivalent \cite[Proposition 11.1.16]{Brown}, in particular $H_n(|KX|)$ and  $H_n(X)$ are isomorphic for all $n$, see \cite[Theorem 7.6.25]{Spanier}.

\subsection*{Discrete Cubical Sets of Graphs}
Another  cubical set with connections arises from the development of a discrete homology theory for metric spaces \cite{BBLL,BCW}. For a given metric space $X$, the singular $(n,r)$-cubes are defined to be the $r$-Lipschitz maps from the $n$-dimensional Hamming cube to the metric space $X$, and the (discrete) homology groups of the metric space $X$ are defined to be the singular homology groups of the resulting singular chain complex.

In a recent paper \cite{BGJW} we study the theory from
\cite{BCW} in the combinatorially interesting case where
the singular $n$-cubes are the graph homomorphisms from the
$n$-dimensional Hamming cube to a given undirected, simple graph $G$. This results in a cubical set $KG$ which is used to define a (discrete) cubical homology of the graph $G$.

For $n \geq 0$, let $Q_n$ be the Hamming $n$-dimensional cube, that is, $Q_n := \{(x_1,\dots,x_n) : x_i \in \{0,1\}, i \in[n]\}$. Define $KG_n$ to be the set of all graph homomorphisms $\sigma: Q_n \longrightarrow G$.
For each $i \in [n]$ and $\sigma \in KG_n$, define face maps $f_i^+\sigma ,f_i^-\sigma \in KG_{n-1}$ such that, for $(a_1,\dots,a_{n-1}) \in Q_{n-1}$,
\begin{align*}
(f_i^+\sigma)(a_1,\dots,a_{n-1}) & := \sigma(a_1,\dots,a_{i-1},1, a_{i},\dots,a_{n-1}), \\
(f_i^-\sigma)(a_1,\dots,a_{n-1}) &:= \sigma(a_1,\dots,a_{i-1},0, a_{i},\dots,a_{n-1}).
\end{align*}
Also, define $\varepsilon_i\sigma \in  KG_{n+1}$ such that, for $(a_1,\dots,a_{n+1}) \in Q_{n+1}$,
\[
(\varepsilon_i\sigma)(a_1,\dots,a_{n+1}) := \sigma(a_1,\dots,a_{i-1},a_{i+1},\dots,a_{n+1}).
\]
Furthermore, for each $i\in [n]$, define connection maps $\Gamma_i^+\sigma, \Gamma_i^-\sigma \in KG_{n+1}$ such that
\[
\Gamma_i^{\varepsilon}\sigma(a_1,\dots, a_{n+1}) := \sigma(a_1,\dots,a_{i-1}, m_\varepsilon(a_i,a_{i+1}),a_{i+2},\dots, a_{n+1}),
\]
where
\[m_\varepsilon(x,y)=
\left\{ \begin{array}{ll}
         \min(x,y) & \mbox{if }\varepsilon = +;\\
        \max(x,y) & \mbox{if } \varepsilon = -.\end{array} \right.
\]
The proof of the following lemma is straightforward and is similar to that of $KX$ being a cubical set with connections.
\begin{Lemma}
The collection $KG := \{KG_n\}_{n\geq 0}$ along with the face maps $f_i^\pm$, degeneracy maps $\varepsilon_i$ and connections $\Gamma_i^\pm$ is a cubical set with connections.
\end{Lemma}

Even though we were able to compute the homology groups of many classes of graphs \cite[Sections 4 and 7]{BGJW}, in general such computations are not feasible and, once again, the need for better understanding of the cubical set itself is evident. Investigating the role of the connections in the nontrivial cycles in the homology groups of $KG$ seems a natural step.

%%%%%%%%%%%%%%%%%%%%%%%%%%%%%%%%%%%%%%%%%%%%%%%%%%%%%%%%%%%%%%%
\section{Homology of the Connection Chain Subcomplex}
% see Antolini's papers.

Let $K$ be a cubical set with connections and let $\CC(K)$ be its non-degenerate chain complex. It is easy to see that the set of connections of $K$ does not form a cubical subset of $K$ as not all faces of a connection are necessarily  connections. However, we will show in this section that the connections generate a chain subcomplex of $\CC(K)$. Furthermore, the homology groups of this subcomplex are trivial.

\begin{Theorem} \label{thm:bd}
Let $K$ be a cubical set and $\CC(K)$ be its chain complex as above. Let $\tau \in K_n$ be a singular $n$-cube and $\beta \in\{+,-\}$. Then,

\begin{enumerate}[(i)]
{\setstretch{1.3}
\item %\begin{align}\nonumber
      $\partial_{n+1} \Gamma_1^\beta(\tau) =  -  \Gamma_{1}^\beta   \sum_{i= 2}^n (-1)^i (f_i^- - f_i^+) (\tau).$
    %\end{align}
\item %\begin{align}\nonumber
       $\partial_{n+1} \Gamma_n^\beta(\tau) = \Gamma_{n-1}^\beta \sum_{i=1}^{n-1} (-1)^i (f_i^- - f_i^+) (\tau).$
    %\end{align}
\item For any $1 < t <n$,
    %\begin{align} \nonumber
    \[\partial_{n+1} \Gamma_t^\beta(\tau) = \Gamma_{t-1}^\beta \sum_{i=1}^{t-1} (-1)^i (f_i^- - f_i^+) (\tau) -  \Gamma_{t}^\beta   \sum_{i= t+1}^n (-1)^i (f_i^- - f_i^+) (\tau).\]
    %\end{align}
}
\end{enumerate}
\end{Theorem}

\begin{proof}
Let $\tau \in K_n$ be a singular $n$-cube and $\beta \in \{+,-\}$.  Then, for $t\in[n]$,
\begin{align*}
\partial_{n+1} \Gamma_t^\beta(\tau) &= \sum_{i=1}^{n+1} (-1)^i(f_i^- - f_i^+) ( \Gamma_t^\beta(\tau)).
\end{align*}
By \ref{def:cub-conn}(\ref{def:ccc3}), %and \ref{def:ccc4},
$f_t^\alpha \Gamma_t^\beta(\tau) = f_{t+1}^\alpha \Gamma_t^\beta(\tau)$
and  $f_i^\alpha \Gamma_t^\beta = \left\{ \begin{array}{ll}
                                \Gamma_{t-1}^\beta f_i^\alpha & \mbox{if  } i < t;\\
                                \Gamma_{t}^\beta f_{i-1}^\alpha & \mbox{if } i > t+1.
                                \end{array} \right.$ 	

 Now \ref{thm:bd}(\emph{i}), i.e. when $t=1$, and \ref{thm:bd}(\emph{ii}), i.e. when $t=n$, follow immediately.
For $1 < t < n$, the following computation implies
\ref{thm:bd}(\emph{iii}),
\begin{align*}
\partial_{n+1} \Gamma_t^\beta(\tau) &=  \Gamma_{t-1}^\beta \sum_{i=1}^{t-1} (-1)^i (f_i^- - f_i^+) (\tau) +  \Gamma_{t}^\beta   \sum_{i=t+2}^{n+1} (-1)^i (f_{i-1}^- - f_{i-1}^+) (\tau) \\
&= \Gamma_{t-1}^\beta \sum_{i=1}^{t-1} (-1)^i (f_i^- - f_i^+) (\tau) -  \Gamma_{t}^\beta   \sum_{i= t+1}^n (-1)^i (f_i^- - f_i^+) (\tau).
\end{align*}
\end{proof}
% Connections generate a chain subcomplex of the singular

Let $K$ be a cubical set with connections and let $\CC(K)$ be its non-degenerate chain complex.
For $n\geq 0$, let $\Con_{n+1}(K)$ be the $R$-submodule of $\CC_{n+1}(K)$ that is generated by the cosets of
$\Gamma_i^\beta(\tau) \mbox{ where } \tau \in K_n, i \in [n], \textrm{ and } \beta \in \{+,-\}$.

The following is an immediate consequence of
\ref{thm:bd}.

\begin{Corollary}
  \label{cor:subcomplex}
   Let $\theta \in  \Con_{n+1}(K)$ then $\partial_{n+1}(\theta) \in \Con_{n}(K)$. In particular,
$\Con(K) = (\Con_\bullet, \partial_\bullet)$ is a chain subcomplex of the chain complex $\CC(K)$. 
\end{Corollary} 

We call $\Con(K)$ the \textit{connection chain complex} of $K$.

Clearly, $\Con_{n+1}(K)$ is generated by the cosets of $\Gamma_i^\beta(\tau)$ where $\tau$ is a non-degenerate singular $n$-cubes. In particular, $\Con_1(K) = (0)$.

%%%%%% %%%%%%%  %%%%%%%
\begin{Corollary} \label{cor:g1}
Let $\tau \in \CC_n(K)$ be a singular $n$-cube. Then, for $\beta \in\{+,-\}$ and $t\in[n]$,
the following equations are true.
\begin{enumerate}[(i)]
\item \begin{align}\nonumber
\partial_{n+1}  \Gamma_1^\beta(\tau) + \Gamma_1^\beta  \partial_n (\tau)
=   \Gamma_{1}^\beta (f_i^+- f_i^-) (\tau).
\end{align}
\item  For $2 \leq t \leq n$,
\begin{align}\nonumber
\partial_{n+1}  \Gamma_t^\beta(\tau) + \Gamma_t^\beta  \partial_n (\tau)
= \Gamma_{t-1}^\beta  \sum_{i=1}^{t-1} (-1)^i (f_i^- - f_i^+) (\tau) +  \Gamma_{t}^\beta  \sum_{i=1}^{t} (-1)^i (f_i^- - f_i^+) (\tau).
\end{align}
\item For $t \in [n]$,
\begin{align}\nonumber
\partial_{n+1}  \sum_{j=1}^{t} (-1)^j \Gamma_j^\beta(\tau) + \sum_{j=1}^{t} (-1)^{j} \Gamma_{j}^\beta  \partial_n(\tau)= (-1)^{t} \Gamma_t^\beta  \sum_{j=1}^t (-1)^i (f_j^- - f_j^+)(\tau).
\end{align}

\item For $t=n \geq 2$,
\begin{align}\nonumber
\partial_{n+1}  \sum_{j=1}^{n} (-1)^j \Gamma_j^\beta(\tau) + \sum_{j=1}^{n-1} (-1)^{j} \Gamma_{j}^\beta  \partial_n(\tau) = 0.
\end{align}
\end{enumerate}
\end{Corollary}

\begin{proof}
\ref{cor:g1}(i) follows by adding the term  $\Gamma_t^\beta  \partial_n (\tau)$ to both sides of \ref{thm:bd}(\emph{iii}).  \ref{cor:g1}(i) is a special case of \ref{cor:g1}(ii) without the first sum on the right hand side. Using alternating summation, \ref{cor:g1}(iii) follows from \ref{cor:g1}(ii). Finally, \ref{cor:g1}(iv) is the case $t=n$ of \ref{cor:g1}(iii).
\end{proof}

\begin{Corollary}\label{cor:gx}
Let $\theta = \Gamma_t^\beta(\tau)$ where $\tau \in K_{n-1}$, $t\in [n-1]$ and $\beta \in\{+,-\}$.  The following equations are true.
\begin{enumerate}[(i)]
\item  \begin{align} \nonumber
\partial_{n+1}  \Gamma_t^\beta(\theta) + \Gamma_t^\beta  \partial_n (\theta) = (-1)^{t+1}\beta \theta + 2\Gamma_{t}^\beta \Gamma_{t-1}^\beta \sum_{i=1}^{t-1}(-1)^i(f_i^- - f_i^+)(\tau).
\end{align}
\item \begin{align} \nonumber
\partial_{n+1}  \sum_{j=1}^{t} (-1)^j \Gamma_j^\beta(\theta) + \sum_{j=1}^{t} (-1)^{j} \Gamma_{j}^\beta  \partial_n(\theta)= -\beta\theta + (-1)^{t} \Gamma_t^\beta \Gamma_{t-1}^\beta  \sum_{i=1}^{t-1} (-1)^i (f_i^- - f_i^+)(\tau).
\end{align}
\end{enumerate}
\end{Corollary}

\begin{proof}
%Let $\theta = \Gamma_t^\beta(\tau)$, for some $\tau \in K_{n-1}$, $t \in [n-1]$ and $\beta \in\{+,-\}$.
%\todo[color=green]{I think this proof should be formatted as paragraphs, not items.}
We know from \ref{cor:g1}(ii) that
\begin{align*}
\partial_{n+1}  \Gamma_t^\beta(\theta) + \Gamma_t^\beta  \partial_n (\theta)
&= \Gamma_{t-1}^\beta  \sum_{i=1}^{t-1} (-1)^i (f_i^- - f_i^+) (\theta) +  \Gamma_{t}^\beta  \sum_{i=1}^{t} (-1)^i (f_i^- - f_i^+) (\theta).
\end{align*}
%By \ref{def:cub-conn}(iii),  %Equation \ref{def:ccc4},
By \ref{def:cub-conn}(\ref{def:ccc3}),
the coset $\Gamma_t^\beta[(-1)^{t} (f_{t}^- - f_{t}^+) (\theta)] = (-1)^{t+1}\beta \theta$ and 
$(f_i^- - f_i^+) (\Gamma_t^\beta(\tau)) =\Gamma_{t-1}^\beta( (f_{i}^- - f_{i}^+)( \tau)$.
Thus
\begin{align*}
\partial_{n+1}  \Gamma_t^\beta(\theta) + \Gamma_t^\beta  \partial_n (\theta)
%&= (-1)^{k+1}\beta \theta +
% (\Gamma_{k-1}^\beta+\Gamma_{k}^\beta)  \sum_{i=1}^{k-1} (-1)^i (f_i^- - f_i^+) (\theta) \\
&= (-1)^{t+1}\beta \theta +
 (\Gamma_{t-1}^\beta+\Gamma_{t}^\beta) \Gamma_{t-1}^\beta  \sum_{i=1}^{t-1} (-1)^i (f_i^- - f_i^+) (\tau) \\
 &= (-1)^{t+1}\beta \theta +
 2\Gamma_{t}^\beta \Gamma_{t-1}^\beta  \sum_{i=1}^{t-1} (-1)^i (f_i^- - f_i^+) (\tau),
\end{align*}
since $\Gamma_{t-1}^\beta\Gamma_{t-1}^\beta = \Gamma_{t}^\beta \Gamma_{t-1}^\beta$. This concludes the proof of \ref{cor:gx}(i).
Now \ref{cor:gx}(ii) follows directly from  \ref{cor:g1}(iii), namely,
 \begin{align*}
\partial_{n+1}  \sum_{j=1}^{t} (-1)^j \Gamma_j^\beta(\theta) + \sum_{j=1}^{t} (-1)^{j} \Gamma_{j}^\beta  \partial_n(\theta)
&= (-1)^{t} \Gamma_t^\beta  \sum_{j=1}^t (-1)^i (f_j^- - f_j^+)(\Gamma_t^\beta(\tau)) \\
&= -\beta\theta + (-1)^{t} \Gamma_t^\beta \Gamma_{t-1}^\beta  \sum_{i=1}^{t-1} (-1)^i (f_i^- - f_i^+)(\tau).
\end{align*}
\end{proof}

\begin{Lemma} \label{lem-main}
Let $\theta = \Gamma_t^\beta(\tau)$, for some $\tau \in K_{n-1}$, $t\in [n-1]$ and $\beta \in\{+,-\}$.  Then
\begin{align*} \label{chain-eqn}
 \partial_{n+1}  [(-1)^{t+1} \Gamma_t^\beta - 2 \sum_{j=1}^{t-1} (-1)^j
 \Gamma_j^\beta](\theta) + [(-1)^{t+1} \Gamma_t^\beta - 2 \sum_{j=1}^{t-1} (-1)^j
 \Gamma_j^\alpha] \partial_n (\theta)
=\beta \theta.
\end{align*}
\end{Lemma}

\begin{proof}
%Equation \ref{chain-eqn}
Follows directly from \ref{cor:gx}. By multiplying the equation from  \ref{cor:gx}(i) by $(-1)^{t}$ and subtracting from that twice the equation from \ref{cor:gx}(ii), we get
\begin{align*}
 \partial_{n+1}  [(-1)^{t} \Gamma_t^\beta - 2 \sum_{j=1}^{t} (-1)^j
 \Gamma_j^\beta](\theta) + [(-1)^{t} \Gamma_k^\beta - 2 \sum_{j=1}^{t} (-1)^j
 \Gamma_j^\alpha] \partial_n (\theta)
=\beta \theta.
\end{align*}
Hence
\begin{align*}
 \partial_{n+1}  [(-1)^{t+1} \Gamma_t^\beta - 2 \sum_{j=1}^{t-1} (-1)^j
 \Gamma_j^\beta](\theta) + [(-1)^{t+1} \Gamma_t^\beta - 2 \sum_{j=1}^{t-1} (-1)^j
 \Gamma_j^\alpha] \partial_n (\theta)
=\beta \theta.
\end{align*}
\end{proof}

%##################

\begin{Remark}
  Notice that it is possible for a singular $n$-cube $\theta$ which is a connection to be written using different connection maps, say $\theta=\Gamma_t^\alpha (\sigma)= \Gamma_s^\beta(\tau)$ for some $t,s \in [n]$, $\alpha, \beta \in\{+,-\}$, and $\sigma, \tau \in K_{n-1}$.
  
If $s=t$ or $s=t+1$, however, then either $\beta=\alpha$ (and hence $\sigma = \tau$) or $\theta$ is degenerate.
Thus if $\theta$ is a non-degenerate singular $n$-cube that is a connection, then $\theta$ can be
written uniquely as $\theta=\Gamma_t^\alpha (\sigma)$ where $t$ is the smallest such index.
The following lemma follows.
\end{Remark}

\begin{Lemma}
Let $\theta$ be a non-degenerate connection $n$-cube, and suppose that $\theta = \Gamma_t^\alpha(\sigma) = \Gamma_s^\beta(\tau)$ where $t \leq s$. Then either
\begin{enumerate}[(i)]
\item $s=t$ or $s=t+1$, and hence $\sigma = \tau$ and $\alpha = \beta$, or
\item $s  > t+1$,  and in this case $\theta = \Gamma_t^\alpha(\Gamma_{s-1}^\beta(\delta)) = \Gamma_s^\beta(\Gamma_t^\alpha(\delta))$ where $\delta= f_t^\alpha(\tau)= f_{t+1}^\alpha(\tau)=f_s^\beta(\sigma)=f_{s-1}^\beta(\sigma)$.
\end{enumerate}
\end{Lemma}

  For the rest of this section, whenever we write a 
  non-degenerate connection $n$-cube $\theta$ as 
  $\theta = \Gamma_t^\beta (\tau)$ 
  we assume $t$ is the smallest index for which such a 
  representation exists. 

  Let $\theta = \Gamma_t^\beta(\tau)$ be a non-degenerate connection. Define
  \[
   \phi_n(\theta) =  \beta \Big[(-1)^{t+1} \Gamma_t^\beta(\theta) - 2 \sum_{j=1}^{t-1} (-1)^j \Gamma_j^\beta(\theta)\Big].
  \]
  The map $\phi_n$ extends linearly to a map $\phi_n :\Con_{n}(K)\longrightarrow \Con_{n+1}(K)$
  such that
  \[
   \phi_n(\sum_{j=1}^s r_{i_j}\Gamma_{i_j}^{\beta_j}(\sigma_j)) = \sum_{j=1}^s r_{i_j}\phi_n(\Gamma_{i_j}^{\beta_j}(\sigma_j)).
  \]

\begin{Theorem}
  \label{thm:homotopy} 
  For any $\theta \in \Con_n(K)$,
  \[
   \partial_{n+1}  \phi_n(\theta) + \phi_{n-1}  \partial_n (\theta) =\theta.
  \]
\end{Theorem}
\begin{proof}
 Recall that $\Con_n(K)$ is freely generated by the cosets of 
 $\theta = \Gamma_t^\beta(\tau)$ where $\tau$
 is non-degenerate and $\beta = +,-$.  Using \ref{lem-main}, to conclude the proof 
 we just need to show that
 \[
  \phi_{n-1}  \partial_n (\theta) = \beta \Big[(-1)^{t+1} \Gamma_t^\beta - 2 \sum_{j=1}^{t-1} (-1)^j
  \Gamma_j^\alpha\Big] \partial_n (\theta).
 \]
 %This argument follows from the fact that $\Gamma_k \Gamma_{k-1} = \Gamma_{k-1} \Gamma_{k-1}$. The details are next.
 Recall that
 $$\partial_{n}(\theta) = \Gamma_{t-1}^\beta   \sum_{i=1}^{t-1} (-1)^i (f_i^- - f_i^+) (\tau) -  \Gamma_{t}^\beta   \sum_{i= k+1}^{n-1} (-1)^i (f_i^- - f_i^+) (\tau).$$
 Now
 \begin{eqnarray*}
  \phi_{n-1}  \partial_n(\theta)
  &=&  \phi_{n-1}   \Gamma_{t-1}^\beta   \sum_{i=1}^{t-1} (-1)^i (f_i^- - f_i^+) (\tau) -\phi_{n-1}  \Gamma_{t}^\beta   \sum_{i= t+1}^{n-1} (-1)^i (f_i^- - f_i^+) (\tau) \\	
  &=&  \beta\Big[(-1)^{t} \Gamma_{t-1}^\beta - 2 \sum_{j=1}^{t-2} (-1)^j \Gamma_j^\beta\Big]  \Gamma_{t-1}^\beta   \sum_{i=1}^{t-1} (-1)^i (f_i^- - f_i^+) (\tau)\\
  & & - \beta\Big[(-1)^{t+1} \Gamma_{t}^\beta - 2 \sum_{j=1}^{t-1} (-1)^j \Gamma_j^\beta\Big] \Gamma_{t}^\beta   \sum_{i= t+1}^{n-1} (-1)^i (f_i^- - f_i^+) (\tau) \\		
  &=&  \beta\Big[(-1)^{t+1} \Gamma_{t-1}^\beta - 2 \sum_{j=1}^{t-1} (-1)^j \Gamma_j^\beta\Big]g]  \Gamma_{t-1}^\beta   \sum_{i=1}^{t-1} (-1)^i (f_i^- - f_i^+) (\tau)\\
  & &- \beta\Big[(-1)^{t+1} \Gamma_{t}^\beta - 2 \sum_{j=1}^{t-1} (-1)^j \Gamma_j^\beta\Big] \Gamma_{t}^\beta   \sum_{i= t+1}^{n-1} (-1)^i (f_i^- - f_i^+) (\tau) \\	
  &=&  \beta\Big[(-1)^{t+1} \Gamma_t^\beta - 2 \sum_{j=1}^{t-1} (-1)^j \Gamma_j^\beta\Big]  \Gamma_{t-1}^\beta   \sum_{i=1}^{t-1} (-1)^i (f_i^- - f_i^+) (\tau)\\
  & &- \beta\Big[(-1)^{t+1} \Gamma_t^\beta - 2 \sum_{j=1}^{t-1} (-1)^j \Gamma_j^\beta\Big] \Gamma_{t}^\beta   \sum_{i= t+1}^{n-1} (-1)^i (f_i^- - f_i^+) (\tau) \\
  &=&  \beta\Big[(-1)^{t+1} \Gamma_t^\beta - 2 \sum_{j=1}^{t-1} (-1)^j
  \Gamma_j^\beta\Big] \partial_n (\theta).
\end{eqnarray*}
\end{proof}

\begin{Corollary}
\label{cor:connectionzero}
 The map $\phi$ is a chain homotopy between the identity and 
 zero chain maps. In particular, we have 
 $\Hom_n(\Con(K))=0$ for all $n$.
\end{Corollary}

\begin{Corollary}
\label{cor:final}
The short exact sequence of chain complexes
\[
0 \longrightarrow \Con_n(K) \hookrightarrow \mathcal{C}_{n+1}(K) \twoheadrightarrow \mathcal{C}_{n+1}(K)/\Con_{n}(K) \longrightarrow 0
\]
 induces a long exact sequence of homology groups, and since $\Hom_n(\Con(K))$
 is trivial, we have
 $\Hom_n(\CC(K)) \cong \Hom_n(\CC(K)/\Con(K))$.
\end{Corollary}

It is well-known that, over a suitable category, the category of chain complexes and the category of crossed complexes are equivalent \cite{brown2003}.  It would be interesting to see whether the results in this paper can be properly stated and extended to the context of crossed complexes.

\section*{Appendix: Homology of Cubical Sets and Homology of Their Geometric Realization}

Recall that $I^n$ is the geometric $n$-dimensional cube $[0,1]^n$. Let $(f_i^\alpha)^\ast : 
I^{n-1} \rightarrow I^n$ be the map sending $(x_1,\ldots,x_{n-1}) \in I^{n-1}$ to 
$(x_1,\ldots, x_{i-1},y,x_{i},\ldots, x_{n-1})$ where $y = 0$ if $\alpha = -$ and
$y =1$ if $\alpha = +$. Let further $(\varepsilon_i)^\ast  : I^n \rightarrow I^{n-1}$ be the
map sending $(x_1,\ldots, x_n)$ to $(x_1,\ldots, x_{i-1},x_{i+1},\ldots, x_n)$.
The geometric realization $|K|$ of a cubical set is the quotient space of the
disjoint union $\coprod I^n \times K_n$ by the equivalence relation $\sim$, which is generated by the following elementary equivalences: 
For $(x_1,\ldots, x_n) \in I^n$ and $\sigma \in K_{n-1}$ we set 
\begin{eqnarray}
  \label{eq:geo1}
((x_1,\ldots, x_n),\varepsilon_i(\sigma)) & \sim & ((\varepsilon_i)^\ast((x_1,\ldots, x_n)),\sigma)
\end{eqnarray}

and, for $(x_1,\ldots, x_{n-1}) \in I^{n-1}$ and $\sigma \in K_{n}$, we set

\begin{eqnarray} 
  \label{eq:geo2}
  ((x_1,\ldots, x_{n-1}),f_i^\alpha (\sigma)) &\sim& ((f_i^\alpha)^\ast((x_1,\ldots, x_{n-1})),\sigma).
\end{eqnarray}

Then $|K|$ can be given the structure of a CW-complex 
whose (open) $n$-cells are the images $e_\sigma^{(n)}$ of the cells $\mathring{I^n} \times \{\sigma\}$ in $|K|$ for
$\sigma \in K^{nd}_n$. Here $K^{nd}_n$ 
denotes the set of non-degenerate $n$-cubes in $K$, see \cite[Remark 11.1.14]{Brown}.
Let $S(K)$ be the cellular chain complex of $|K|$. 
By the definition of $S(K)$ the cells $e_\sigma^{(n)}$
for $\sigma \in K_n^{nd}$ form a basis of its $n$\textsuperscript{th} chain group $S_n(K)$. 
It is well known (see \cite[Corollary 3.9.11]{Federer}) that identifying $\sigma \in K_n^{nd}$ with $e_\sigma^{(n)}$ 
yields the following isomorphism of chain complexes.

\begin{Lemma}[Corollary 3.9.11 \cite{Federer}]
  $\CC(K) \cong S(K)$. 
\end{Lemma}

If the cubical set $K$ is a cubical set with connections then there is an associated 
geometric realization $|K|'$ which is the quotient of the 
disjoint union $\coprod I^n \times K_n$ by the equivalence relation $\sim'$, which is generated by \eqref{eq:geo1}, \eqref{eq:geo2} and the
relation 
\begin{eqnarray} 
  \label{eq:geo3}
  ((\Gamma_i^\alpha)^\ast ((x_1,\ldots, x_{n})),\sigma) & \sim' & ((x_1,\ldots, x_{n}),\Gamma_i^\alpha(\sigma))
\end{eqnarray}
for $\sigma \in K_{n-1}$ and $(x_1,\ldots, x_n) \in I^n$. Here 
$(\Gamma_i^\alpha)^\ast : I^{n} \rightarrow I^{n-1}$ is defined by
$$(\Gamma_i^\alpha)^\ast ((x_1,\ldots, x_{n})) = 
\left\{ \begin{array}{cc} (x_1,\ldots,x_{i-1},\max(x_i,x_{i+1}),x_{i+2} ,\ldots, x_n) & 
\mbox{~if~} \alpha = - \\
(x_1,\ldots,x_{i-1},\min(x_i,x_{i+1}),x_{i+2} ,\ldots, x_n) & 
\mbox{~if~} \alpha = +
\end{array} \right. .$$
In particular, $\sim'$ is  coarser than $\sim$ and hence $|K|'$ can be seen as
a quotient of $|K|$ by the additional identifications implied by \eqref{eq:geo3}.
Let $K_n^{ndc}$ be the set of $n$-cubes in $K$ that are neither degenerate nor
connections. 

In order to understand the relation between $|K|$ and $|K|'$ we 
need to understand the face structure of cubes in $K_n^{ndc}$.
 For that we consider for any cube $\sigma\in K_n$ the set of all of its faces $\tau$; i.e. all 
 cubes $\tau$ such that
 $\tau = f_{i_1}^{\alpha_1} ( \cdots (f_{i_r}^{\alpha_r}(\sigma))\cdots )$ for a choice of
 $i_1,\ldots, i_r$ and $\alpha_1,\ldots, \alpha_r$. 
 For $\sigma \in K$ we denote by $F_\sigma$ the set of
 its faces. 
 We order the cubes from $K$ by saying that $\tau$ is
 smaller than $\sigma$ if $\tau$ is a face of $\sigma$.
 With this notation we are in position to formulate the
 following structural result on the role of non-degenerate
 and non-connection cubes in the face structure.
 
 \begin{Lemma}
    \label{lem:maxcube}
        For any $\tau \in K_{n}$ there is a unique face $\rho$ of $\tau$ that is maximal
    with the property that it is neither degenerate nor a 
    connection. Moreover, if $\tau = \varepsilon_i(\sigma)$ 
    or $\tau = \Gamma_i^\alpha(\sigma)$ then $\rho$ is a subface of $\sigma$ and $\tau = g_k \cdots g_1 (\rho)$ for suitably chosen connection and degeneracy maps $g_1,\ldots, g_k$ for some $k \geq 0$. 
\end{Lemma}

\begin{proof}
  We prove the assertion by induction on 
  the dimension $n$. 
  
  If $n = 0$ then $\tau$ is non-degenerate and
  non-connection. Hence $\tau$ itself is the
  maximal face we are looking for. 
  
  Let $n > 0$. If $\tau$ is neither degenerate nor a connection then again $\tau$ itself is the 
  unique maximal face. 
  
  Let $\tau$ be degenerate, say $\tau =
  \varepsilon_i(\sigma)$ for some $i \in [n]$ and
  some $(n-1)$-cube $\sigma$.
  Then, by (iii) of \ref{def:cubicalset}, $f_j^\alpha(\tau) = \sigma$ if $i = j$, and
  $f_j^\alpha(\tau) = \varepsilon_{i-1}(f_j^\alpha(\sigma))$ if $j < i$ and $= \varepsilon_{i}(f_{j-1}^\alpha(\sigma))$ if $j > i$. By induction, we
  know that there is an unique maximal non-degenerate
  and non-connection face $\rho$ 
  of $\sigma$. We claim that $\rho$ is the unique maximal non-degenerate and non-connection face of $\tau$. By induction we know that each  $\varepsilon_r(f_s^\beta(\sigma))$ has a unique maximal non-degenerate, non-connection face which is a subface of $f_s^\beta(\sigma)$ and 
  hence of
  $\sigma$. In particular, they must be subfaces of $\rho$.
  If follows by induction that $\sigma = g_k\cdots g_1 \rho$ 
  for a sequence of degeneracy and connection maps $g_1,\ldots, g_k$ 
  and $k \geq 0$. Then $\tau = \varepsilon_i g_k
  \cdots g_1 \rho$.
  
  Finally, consider the case that $\tau$ is a connection.
  Say $\tau = \Gamma_{i}^\alpha(\sigma)$ for some $i \in[n]$ and some $(n-1)$-cube $\sigma$.
  Notice that, by (iii) of \ref{def:cub-conn}, every $(n-1)$-face of $\tau$ other than $\sigma$ is either  $\Gamma_j^\beta(f_t^\alpha(\sigma))$ or $\varepsilon_j(f_t^\alpha(\sigma))$ for some $j\in$, $t\in[n]$, and $\alpha \in\{+,-\}$. By induction 
$\sigma$ and any $\Gamma_j^\beta(f_t^\alpha(\sigma))$
have an unique maximal non-degenerate, non-connection face. Again by induction the latter are subfaces of $\sigma$. 
In particular, they must be subfaces of the unique 
maximal non-degenerate, non-connection face $\rho$ of 
$\sigma$. From the induction hypothesis it
follows $\sigma = g_k\cdots g_1 \rho$ 
  for a sequence of degeneracy and connection maps $g_1,\ldots, g_k$ 
  and $k \geq 0$. Then $\tau = \Gamma_i^\alpha g_k
  \cdots g_1 \rho$.
\end{proof}

Note that along the same lines one can show that for any cube there
is a unique maximal non-degenerate face.

The relations among the degeneracy and connection 
maps allow the following strengthening of \ref{lem:maxcube}.

\begin{Lemma}
  \label{lem:maxcubespecial}
  For any $\tau \in K_{n}$ there is a unique face $\rho$ of $\tau$ that is maximal
  with the property that it is neither degenerate nor a 
  connection. Moreover, if $\tau$ is non-degenerate then
  $\tau = g_k \cdots g_1 (\rho)$ for suitably chosen connection maps $g_1,\ldots, g_k$ and some $k \geq 0$. 
\end{Lemma}
\begin{proof}
  From \ref{lem:maxcube} it follows  
  that there is a unique maximal face $\rho$ of $\tau$ that is neither degenerate nor a connection. It also
  follows from that lemma that 
  $\tau = g_k\cdots g_1 \rho$, for degeneracy and
  connection maps $g_1,\ldots, g_k$. If all 
  $g_i$ are connection maps we are done. Assume there is an $i$ 
  such that $g_i$ is a degeneracy map. We claim that then 
  $\tau$ is degenerate. We prove
  the claim by downward induction on the maximal $i$ such that
  $g_i$ is a degeneracy map. 
  If $i = k$ then $\tau$ is degenerate, contradicting the assumptions. If $i <k$ then by \ref{def:cub-conn}(iii)
  there is a connection or degeneracy map $g_i'$ and s
  degeneracy map $g_{i+1}'$ such that
  $$\tau  = g_k\cdots g_{i+2} 
  g_{i+1}' g_i' g_{i-1} \cdots g_1 \rho.$$
  By induction this implies that
  $\tau$ is degenerate.
\end{proof}

Now we apply the results on the face structure in order to 
understand the attachment of cells in $|K|$ and $|K|'$.
We assume without stating the proofs the following fact:

\begin{itemize}
    \item Let $(x,\sigma),(y,\sigma) \in I^{\dim \sigma}\times \{\sigma\}$. Then 
    $(x,\sigma)$, $(y,\sigma)$
    are identified through the equivalence relation generated by \ref{eq:geo1},\ref{eq:geo2} (resp. \ref{eq:geo1}, \ref{eq:geo2} and \ref{eq:geo3}) on 
    $\coprod_{\tau \in K} I^{\dim \tau} \times \{\tau\}$
    if and only if they are identified by the
    equivalence relation generated by \ref{eq:geo1},\ref{eq:geo2} (resp. 
    \ref{eq:geo1}, \ref{eq:geo2} and \ref{eq:geo3}). 
    on $\coprod_{\tau \in F_\sigma} I^{\dim \tau} 
    \times \{ \tau\}$. 
\end{itemize}

This fact allows us to consider the identifications 
by the equivalence relations we consider as local 
identifications among points in the cells corresponding
to the faces of a given cell.

\begin{Lemma}
  \label{lem:retract}
  Let $\tau \in K_n$ be such that $\tau = g_k \cdots g_1 \rho$
  for some cube $\rho$ and connection maps
  $g_1,\ldots, g_k$. Let $\sim_\tau$ 
  be the restriction of the equivalence relation generated by \eqref{eq:geo1},
  \eqref{eq:geo2}, \eqref{eq:geo3} to
  $M_\tau = \coprod_{\sigma \in F_\tau} I^{\dim \sigma} \times \{\sigma\}$ and define $\sim_\rho$ analogously. 
  Then there is a retraction $p_\tau: 
  M_\tau/\sim_\tau \rightarrow M_\rho/\sim_\rho$.
\end{Lemma}
\begin{proof}
   We construct the retraction by induction on $k$. 
   For $k = 0$ the identity is the desired retraction.
   
   Let $k \ge 1$ and assume that for $\tau' = 
   g_{k-1} \cdots g_1 \rho$ there is such a retraction
   $p_{\tau'} : M_{\tau'}/\sim_{\tau'} 
   \rightarrow M_\rho/\sim_\rho$.
   Then $\tau = g_k \tau'$. The equivalence relation 
   on $I^{\dim \tau} \times \{\tau\}$ induced by the connection map 
   $g_k = \Gamma_i^\beta$ has equivalence classes being
   sets with fixed maximum or
   minimum of the $i$\textsuperscript{th} and 
   $(i+1)$\textsuperscript{st} coordinate
   depending on $\beta$ being $+$ or $-$. Each 
   equivalence class
   has exactly two points that via the face maps
   $f_i^\beta$ and $f_{i+1}^\beta$ are
   identified with points in $I^{\dim \tau'} \times
   \{ \tau'\}$, indeed both points are identified with the
   same point. The map that sends each equivalence class 
   to the image of this point in $M_{\tau'} 
   /\sim_{\tau'}$ provides a retraction
   from $M_\tau/ \sim_\tau$ to $M_{\tau'}/\sim_{\tau'}$.
   Composing this retraction with the retraction
   from $p_{\tau'}$ provides the
   asserted retraction. This concludes the induction
   step.
 \end{proof}

We now introduce the concept of pushing cells for a general CW-complex which we will then match with the process of 
passing from $|K|$ to $|K|'$ in our case.
Let $X$ be a CW-complex where, for $n \geq 0$, 
$X_n = (e_\sigma^{(n)})_{\sigma \in J_n}$ is the 
set of open $n$-cells in $X$ for some indexing set $J_n$. 
For each $\sigma \in J_n$ let $g_\sigma : \partial \overline{e^{(n)}} \rightarrow X^{(n-1)}$ be the attaching map. For some fixed $N \geq 0$, let $\bar{J}_N \subseteq J_N$
be a subset of the index set of the cells in dimension 
$N$ such that, for each $\sigma \in \bar{J}_N$,
\begin{itemize}
    \item there is a 
     $\tau \in J_\ell$ for some $\ell < N$ such that
     $\Im g_\sigma \subseteq \overline{e^{(\ell)}_\tau}$, and 
     \item for this $\tau$ there is a retraction $p_\sigma :
     \overline{e_{\sigma}^{(N)}} \rightarrow 
     \overline{e_\tau^{(\ell)}}$.
\end{itemize}

Now let $X^{\push}$ be the CW-complex with $X_n^{\push} = (\tilde{e}_\sigma^{(n)})_{\sigma\in J_n'}$ the open $n$-cells in $X^{\push}$
 where $J_n' = J_n$ for $n   \neq N$ and 
 $J_N' = J_N \setminus \bar{J}_N$ and 
 attaching maps $g'_\tau(x) = g_\tau(x)$ if $g_\tau(x) \not\in 
 \overline{e_\sigma^{(N)}}$ for some $\sigma \in \bar{J_N}$ and $g'_\sigma (x) = p_\tau (g_\sigma(x))$ otherwise.
 In this situation we say that $X^{\push}$ arises from $X$ by pushing
 the cells $e_\sigma^{(N)}$ for $\sigma \in \bar{J}_N$.
 
 Next we show that $|K|$ and $|K|'$
 are examples of CW-complexes that arise from each other by
 pushing cells.
 
 \begin{Lemma}
  \label{lem:connect}
  The geometric realization $|K|'$ is a
  CW-complex that arises
  from the CW-complex of the geometric realization $|K|$ by pushing the cells 
  corresponding to connections
  successively by dimension in increasing order.
  In particular, $|K|'$ can be given the structure of a CW-complex with $n$-cells indexed
  by the $K_n^{ndc}$.
\end{Lemma}
\begin{proof}
   Since the first connection 
   cells (that are not already degenerate) arise in dimension $2$, we can assume the following situation. For
   some $n \geq 2$ we have
   constructed a complex $X$ such that
   \begin{itemize}
       \item[(a)] $X$ arises from $|K|$ by pushing all 
         cells that correspond to connections of dimensions
         $< n$ where $n \geq 2$.
       \item[(b)] $|K|/\sim_{n} \cong X$ where 
         $\sim_{n}$ is the equivalence relation which 
         has singleton equivalence classes outside the closure of the cells of dimension $<n$ and equals \eqref{eq:geo3} when applied to the union of the closures of all other cells.
   \end{itemize}
   Now let $\sigma \in K_n$ be a connection that is 
   non-degenerate. Then by \ref{lem:maxcube} there is
   a unique maximal face $\tau \in K_{\ell}$ of $\sigma$
   which is non-degenerate and non-connection. Since 
   all proper connection faces of $\sigma$ have been
   pushed the attaching map $g_\sigma$ of the $N$-cell
   $I^N$ corresponding to $\sigma$ has as its image
   the $\ell$-cell corresponding to $\tau$. Furthermore,
   by \ref{lem:maxcubespecial} the conditions of 
   \ref{lem:retract} are satisfied and 
   there is a retraction $p_\sigma$ from then 
   closure of the $N$-cell corresponding to $\sigma$ to 
   the closure of the $\ell$-cell corresponding to 
   $\sigma$. Moreover, by \ref{lem:retract} the map
   $\sigma$ identifies the exactly those elements which
   lie in the same equivalence class of $\sim_n$.
   
   Hence the conditions for a pushing to the cells
   corresponding to non-degenerate connections
   $\sigma$ are satisfied. 
   It follows that (a) and (b) are satisfied for 
   $\sim_n$.
\end{proof}

Finally, we need to understand the impact of pushing cells
on the cellular chain complex of a CW-complex.

 \begin{Lemma}  
   \label{lem:general}
   Let $X$ be a CW-complex with cells 
   $X_n = (e_\sigma^{(n)})_{\sigma \in J_n}$, $n \geq 0$. 
   Assume that there is a dimension $N$ such that
   $X^{\push}$ arises from $X$ by pushing the cells
   $e_\sigma^{(N)}$ for $\sigma \in \bar{J}_N \subseteq
   J_N$. 
   Let 
   $$\partial e_\sigma^{(n)} = \sum_{\sigma' \in J_{n-1}} 
   d_{\sigma,\sigma'} \, e_{\sigma'}^{(n-1)}$$ be the differential of
   the cellular chain complex associated to $X$. Then 
   for $\sigma \in J_n \setminus J_N$, $\sigma' \in J_{n-1}
   \setminus J_N$
   the coefficient $d^{\push}_{\sigma,\sigma'}$ in the 
   differential of the cellular chain complex
   of $X^{\push}$ we have $d^{\push}_{\sigma,\sigma'} = d_{\sigma,\sigma'}$. 
\end{Lemma}
\begin{proof}
  The coefficient $d_{\sigma,\sigma'}$ is given as
  the degree of the composition
  $$S^{n-1} \cong \partial \overline{e^{(n)}} \xrightarrow{g_\sigma} X^{(n-1)} 
  \rightarrow X^{(n-1)} / (X^{(n-1)} \setminus e_{\sigma'}^{(n-1)} )
  \cong S^{n-1}.$$
  The composition depends on the attaching maps
  $g_\sigma$ of the cells corresponding to $\sigma$ 
  only.
  Now consider the same sequence in $X^{\push}$,
  which in particular implies $\sigma,\sigma'
  \neq \tau$. Let $g'_\sigma$ be the corresponding 
  attaching maps.
  If $g_\sigma(x) \not\in \overline{e_\tau^{(N)}}$ for
  some $\tau \in \bar{J}_N$ then $g_\sigma(x) = g'_\sigma(x)$. If 
  $g_\sigma(x) \in \overline{e_\tau^{(N)}}$ for some
  $\tau \in \bar{J}_N$ then $g'_\sigma(x) = p_\sigma (g_\sigma(x))$ for a retraction $p_\sigma$. 
  But in the latter case  
  $g_\sigma(x)$ and $g'_\sigma(x)$ lie in the
  complement of any $(n-1)$ cell different from 
  $e_\tau^{(N)}$.  In that
  situation the composition is again 
  determined by $g_\sigma$.
  It follows that $d_{\sigma,\sigma'} = 
  d^{\push}_{\sigma,\sigma'}$.
\end{proof}

By definition $\CC_n(K)/\Con_n(K)$ has a basis indexed by $K_n^{ndc}$. The differential of the complex $\CC_n(K)/\Con_n(K)$ 
are arises from the differential in $\CC(K)$ in the
following way. Let $\partial \alpha$ is the differential of $\alpha \in K^{ndc}_n$ in $\CC_n(K)$ then we set
all coefficients of element from $K_{n-1}^{nd} \setminus 
K_{n-1}^{ndc}$ to $0$.
Now the following theorem is an immediate consequence of \ref{lem:general} and \ref{lem:connect}.

\begin{Theorem}
  \label{thm:final}
  The cellular chain complex $S'(K)$ of $|K|'$ is isomorphic to the 
  quotient complex $\CC(K)/\Con(K)$. In particular,
  $$H_i(|K|') \cong H_i(S'(K)) \cong H_i( \CC(K)/\Con(K)).$$
\end{Theorem}
\begin{proof}
  The assertion follows immediately from \ref{lem:connect} and \ref{lem:general}.
\end{proof}

The theorem together with \ref{cor:final} implies the 
following.

\begin{Corollary}
\label{cor:superfinal}
Let $K$ be a cubical set with connections. Then 
$$H_i(|K|') \cong H_i(S'(K)) \cong H_i( \CC(K)/\Con(K))
\cong H_i(\CC(K)) \cong H_i(|K|).$$
\end{Corollary} 

This fact provides another motivation for the study of connections.

%%%%%%%%%%%%%%%%%%%%%%%%%%%%%%%%%%%%%%%%%%%%%%%%%%%%%%%%%%%%%%%
\section{Acknowledgment}
The authors thank Professor Ronald Brown for his valuable comments and suggestions on an earlier version of this paper.

%Where the bibliography will be printed
% \printbibliography
%%%%%%%%%%%%%%%%%%%%%%%%%%%%%%%%%%%%%%%%%%%%%%%%%%%%%%%%%%%%%%
\bibliographystyle{siam}
\bibliography{DiscHomo}
%%%%%%%%%%%%%%%%%%%%%%%%%%%%%%%%%%%%%%%%%%%%%%%%%%%%%%%%%%%%%%
\end{document}